\documentclass[article]{elsarticle}

\usepackage{lineno,hyperref}

\usepackage[utf8]{inputenc}
\usepackage{amsmath}
\usepackage{amssymb}
\usepackage{amsthm}
\usepackage{color}
\usepackage{tikz-cd}
\usepackage{xypic}
\usepackage{mathtools}

\newtheorem{defn}{Definition}
\newtheorem{thm}{Theorem}

\newtheorem{lemma}{Lemma}
\newtheorem{rmk}{Remark}

\graphicspath{ {./images/} }

\journal{-}

\begin{document}

\begin{frontmatter}

\title{Loop Homology of Bi-secondary Structures II}


\author[math2,bi]{Andrei C. Bura}
\ead{anbur12@vt.edu}
\author[bii]{Qijun He\corref{mycorrespondingauthor}}
\ead{qhe196@gmail.com}
\author[bii,math]{Christian M. Reidys}
\ead{duck@santafe.edu}


\cortext[mycorrespondingauthor]{Corresponding author}
\address[math2]{Department of Mathematics, Virginia Tech, 225 Stanger Street,
Blacksburg, VA 24061-1026}
\address[bi]{Biocomplexity Institute of Virginia Tech, 1015 Life Sciences Circle 
Blacksburg, VA 24061}
\address[bii]{Biocomplexity Institute and Initiative, University of Virginia, 995 Research Park Boulevard, Charlottesville, VA 22911}
\address[math]{Department of Mathematics, University of Virginia, 141 Cabell Dr, Charlottesville, VA 22903}

\begin{abstract}
  In this paper we further describe the features of the topological space $K(R)$ obtained from the loop nerve of $R$, for $R=(S,T)$ a bi-secondary structure. We will first identify certain distinct combinatorial structures in the arc diagram of $R$ which we will call crossing components. The main theorem of this paper shows that the total number of these crossing components equals the rank of $H_2(R)$, the second homology group of the loop nerve. 
\end{abstract}


\end{frontmatter}


\section{Introduction}

In \citep{bura2019loop} Loop Homology of Bi-secodary Structures, we proved that $H_2(R)$ is free abelian. However, we've yet to identify the combinatorial object within the diagram of the bi-structure $R$ that contributes a generator to $H_2(R)$. In the the following, we will identify the precise sub-structures of a given bi-secondary structure $R$, that when considered within the loop nerve $K(R)$, correspond to sub-complexes that triangulate $2$-spheres. These sub-structures we will call crossing components (CCs). We will show that there is a bijective correspondence between any minimal generating set of $H_2(R)$ and the set of CCs of $R$ and thus, the number of CCs equals the rank of $H_2(R)$.

\section{Secondary and Bi-Secondary Structures}

\begin{defn}
  An RNA diagram $S$ over $[n]$, is a vertex-labeled graph whose vertices are drawn on the horizontal axis
  and labeled by $[n]=\{1,\ldots,n\}$. An \emph{arc} $\mu=(i,j),i<j$, is an ordered pair of vertices, which represents
  the base pairing between the $i$-th and $j$-th nucleotides in the RNA structure. We denote by $b(\mu)=i$ and $e(\mu)=j$ the start and endpoints of an arc $\mu\in S$. Furthermore, each vertex can
  be paired with at most one other vertex, and the arc that connects them is drawn in the upper half-plane. We
  introduce two ``formal'' vertices associated with positions $0$ and $n + 1$, respectively, closing any diagram
  by the arc $(0, n + 1)$, called the rainbow. The set $[0,n+1]$ is called the diagram's \emph{backbone}. 
\end{defn}

\begin{defn}
  Let $S$ be an RNA diagram over $[n]$. Two arcs $(i, j)$ and $(p, q)$ are called \emph{crossing} if and only if
  $i < p < j < q$. $S$ is called a \emph{secondary structure} if it does not contain any crossing arcs. The arcs
  of $S$ can be endowed with a partial order as follows: $(k,l)\prec_S (i,j)\iff i < k < l< j$. We denote this by
  $(S,\prec_S)$ and call it the arc poset of $S$. Finally, an \emph{interval} $[i,j]$ on the backbone is the set
  of vertices $\{i,i+1,\ldots,j-1,j\}$.
\end{defn}

\begin{defn}
  Let $S$ be a secondary structure over $[n]$. A \emph{loop} $s$ in $S$ is a subset of vertices, represented as
  a disjoint union of a sequence of contiguous blocks on the backbone of $S$, $s=\dot\bigcup_{i=1}^k [a_i,b_i]$,
  such that $(a_1,b_k)$ and $(b_i,a_{i+1})$, for $1\leq i\leq k-1$, are arcs and such that any other
  interval-vertices are unpaired. Let $\alpha_s$ denote the unique, maximal arc $(a_1,b_k)$ of
  the loop.
\end{defn}
 
In this paper we shall identify a secondary structure with its set of loops. 

\begin{rmk}
Let $S$ be a secondary structure over $[n]$ and $s=\dot\bigcup_{i=1}^k [a_i,b_i]$ a loop in $S$, then
\begin{itemize}
    \item Each unpaired vertex is contained in exactly one loop.
    \item The arc $(a_1,b_k)$ is maximal w.r.t.~$\prec_S$ among all arcs contained in $s$, i.e.~there is a bijection between
        arcs and loops, mapping each loop to its maximal arc.
    \item The Hasse diagram of the $S$ arc-poset is a rooted tree $\text{Tr}(S)$, having the rainbow arc
        as the root.
    \item Each non-rainbow arc appears in exactly two loops.
\end{itemize}
\end{rmk}

    Let $X=\{x_0,x_1,\ldots,x_m\}$ be a collection of finite sets. We call
  $Y=\{x_{i_0},\ldots,x_{i_d}\}\subseteq X$ a \emph{$d$-simplex} of $X$ iff
  $\bigcap_{k=0}^d x_{i_k}\neq \varnothing$. We set $\Omega(Y)=\bigcap_{k=0}^d x_{i_k}$ and
  denote by $\omega(Y)=|\Omega(Y)|\neq 0$.
  Let $K_d(X)$ be the set of all $d$-simplices of $X$, then the \emph{nerve} of $X$ is
  $$
  K(X)=\dot\bigcup_{d=0}^{\infty} K_d(X)\subseteq 2^{X}.
  $$
  A $d'$-simplex $Y'\in K(X)$ is called a $d'$-face of $Y$ if $d'<d$ and $Y'\subseteq Y$. By construction,
  $K(X)$ is an abstract simplicial complex.

Let $S$ be a secondary structure over $[n]$. The geometric realization of $K(S)$, the nerve over the set
of loops of $S$, is a tree.
\begin{defn}
  Given two secondary structures $S$ and $T$ over $[n]$, we refer to the pair $R=(S,T)$ as a bi-secondary
  structure. Let $S\cup T$ be the loop set of $R$ and $K(R)=\dot\bigcup_{d=0}^{\infty} K_d(R)$ be its nerve
  of loops.
\end{defn}

We represent the arc diagram of a bi-secondary structure $R=(S,T)$ with the arcs of $S$ in the upper half plane while the
arcs of $T$ reside in the lower half plane.

Let $R=(S,T)$ be a bi-secondary structure with loop nerve $K(R)$. A $1$-simplex
$Y=\{r_{i_0},r_{i_1}\}\in K_1(R)$ is called \emph{pure} if $r_{i_0}$ and $r_{i_1}$
are loops in the same secondary structure and \emph{mixed}, otherwise. Any $2$-simplex $Y\in K_2(R)$ had exactly one pure edge and two mixed edges as its $1$-faces (See Part Two: Loop Homology of Bi-secodary Structures).

\begin{defn}
Given $R=(S,T)$, a bi-secondary structure on $[n]$, $R$ is called a \emph{non-overlapping} bi-secondary structure, if any nucleotide $q\in\{1,\ldots,n\}$ has degree at most three in the arc diagram of $R$.
\end{defn}

\section{Decorations and Closures}

\begin{defn}
Let $R=(S,T)$ be a bi-secondary structure with loop set $R=S\cup T$. We define the arc \emph{line graph} of $R$ to be $G=(R,E)$ where $$E\ni e=(s\in S,t\in T)\Leftrightarrow $$$$b(\alpha_s)< b(\alpha_t)<e(\alpha_s)<e(\alpha_t)\text{ or } b(\alpha_t)< b(\alpha_s)<e(\alpha_t)<e(\alpha_s)$$ i.e. the arc $\alpha_s$ intersects the arc $\alpha_t$ if we were to flip $\alpha_t$ to the upper half plane. In this case we say the two arcs $\alpha_s$ and $\alpha_t$ are $\it{crossing}$. We call the set of arcs associated to a non-trivial connected component of this graph, a \emph{crossing component} (CC) of $R$. By non-trivial, we mean the vertex size of such a component must be strictly larger than $1$. When convenient, and when no possibility of confusion exists, we will also identify $X$ with the set of loops whose unique maximal arcs are the elements of $X$. We denote the set of all CCs of $R$ by $\chi(R)$.
\end{defn}

\begin{defn}
Let $R=(S,T)$ be a non-overlapping bi-secondary structure on $[n]$. Let $q\in \{1,\ldots,n\}$ be a nucleotide of degree exactly three in the arc diagram of $R$. Furthermore, let $Y\in K_2(R)$ be a $2$-simplex. We call a copy of $Y$, indexed by $q$ and denoted by $Y_q$, a \emph{decoration} of $Y$ at $q$ if $q\in \Omega(Y)$. We denote by $K_2(R)^*$ the set of all possible decorations of elements of $K_2(R)$.
\end{defn}

\begin{rmk}\label{purearc}
We make the following observations about decorations
\begin{itemize}
    \item Clearly $K_2(R)\xhookrightarrow{} K_2(R)^*$. 
    \item Since $1\le\omega(Y)\le 2$ any $Y\in K_2(R)$ has at most two, and at least one decoration in $K_2(R)^*$. 
    \item Let $Y=[x,y,z]\in K_2(R)$  and assume that $[x,y]$ is the pure edge of $Y$ Note then that $x\le y\le z$ in terms of the simplicial ordering on $K(R) $(See Part Two: Loop Homology of Bi-secodary Structures).Then, for any decoration $Y_q\in K_2(R)^*$ we have $q=b(\alpha_x)$ or $q=e(\alpha_x)$. Hence, to each decoration $Y_q\in K_2(R)^*$ there corresponds a unique arc $\gamma(Y)=\alpha_x$ such that either $b(\gamma(Y))=q$ or $e(\gamma(Y))=q$. We call this arc the \emph{pure arc} of the decoration. (See Figure~\ref{purearcdecoration})
\end{itemize}
\end{rmk}

\begin{figure}[htbp]
    \centering
    \includegraphics[width=0.25\textwidth]{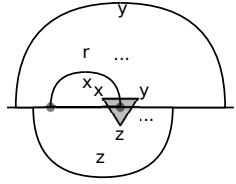}
    \caption{
      A decoration $Y_{e(r)}=[x,y,z]_{e(r)}$, and its pure arc $\gamma(Y)=r=\alpha_x$.}
    \label{purearcdecoration}
\end{figure}

\begin{defn}
  Let $K=\dot\bigcup_{d=0}^\infty K_d$ be an abstract simplicial complex and let $Y\in K_d$ be a
  $d$-simplex. Let $Y'$ be a $k$-face of $Y$, where $k<d$.
  We say $Y'$ is \emph{$Y$-exposed} if and only if any simplex of $K$ that contains $Y'$ as a $k$-face must be a face of $Y$.
\end{defn}

\begin{lemma}\label{exposededge}
Let $R=(S,T)$ be a non-overlapping bi-secondary structure. For any $2$-simplex $Y\in K_2(R)$, if $\gamma(Y)$ is not contained in any CC of $R$, then the pure edge of $Y$ is $Y$-exposed. 
\end{lemma}
\begin{proof}
W.l.o.g., let us assume $Y=[s_0,s_1,t_0]$ with pure arc $\gamma(Y)=\alpha_{s_0}$. Since $R$ is non-overlapping, $b(\gamma(Y))$ and $e(\gamma(Y))$ are unpaired nucleotides in the $T$ secondary structure. Hence, each of $b(\gamma(Y))$ and $e(\gamma(Y))$ are contained in exactly one loop in $T$. Furthermore, as $\gamma(Y)$ does not cross any arc in $T$, then for any arc $z\in T$ we must have that $[b(z)<b(\gamma(Y))<e(z)]\Leftrightarrow [b(z)<e(\gamma(Y))<e(z)]$. Therefore, $b(\gamma(Y))$ and $e(\gamma(Y))$ are contained in the same loop in $T$, namely, $t_0$. Since $t_0$ is the unique loop in $T$ that has nonempty mutual intersection with $s_0$ and $s_1$, $Y=[s_0,s_1,t_0]$ is the unique $2$-simplex in $K(R)$ that contains $[s_0,s_1]$ as an edge. Thus $[s_0,s_1]$ is $Y$-exposed and the lemma follows.
\end{proof}

\begin{lemma}\label{existenceofmixed}
Let $R=(S,T)$ be a bi-secondary structure. For any $3$-simplex $W$ in $K_3(R)$, there exists one mixed edge $Z\in K_1(R)$ that is $W$-exposed.
\end{lemma}
\begin{proof}
Let $W=[s_0,s_1,t_0,t_1]\in K_3(R)$, with $s_0\le  s_1\le t_0\le t_1$ (in terms of the simplicial ordering on $K(R)$). Since $s_0\cap s_1\cap t_0\cap t_1\neq \varnothing$, $\alpha_{s_0}$ and $\alpha_{t_0}$ must share at least one endpoint. W.l.o.g., we distinguish the following two cases (See Figure~\ref{existenceofmixedfigure}):\\
{\it Case $1$:} $b(\alpha_{s_0})<e(\alpha_{s_0})=b(\alpha_{t_0})<e(\alpha_{t_0})$.\\
In this case, we have $s_0\cap t_0=s_0\cap s_1\cap t_0\cap t_1=\{e(\alpha_{s_0})\}$. Suppose there exists another $2$-simplex (triangle) that contains the $1$-simplex (edge) $[s_0,t_0]$. Namely, suppose there exists $x\in R$, with $s_{0,1}\ne x\ne t_{0,1},$ and such that $s_0\cap t_0\cap x\neq \varnothing$. Then $$\varnothing\ne s_0\cap t_0\cap x=s_0\cap s_1\cap t_0\cap t_1\cap x\implies 
\begin{cases}
  s_0\cap s_1\cap x\ne \varnothing, x\in S \\
  t_0\cap t_1\cap x\ne \varnothing, x\in T
  \end{cases}.$$ 
Either case this yields a contradiction, since three loops of the same secondary structure intersect trivially (See Part Two: Loop Homology of Bi-secodary Structures).Thus, it must be the case that $Z=[s_0,t_0]$ is $W$-exposed.\\
\begin{figure}[htbp]
    \centering
    \includegraphics[width=0.5\textwidth]{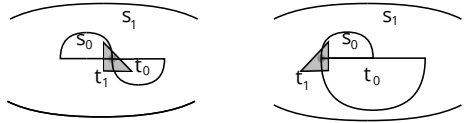}
    \caption{
      LHS: Case 1, $[s_0,t_0]$ is the mixed exposed $W$-edge. RHS: Case 2, $[s_0,t_1]$ is the mixed exposed $W$-edge.}
    \label{existenceofmixedfigure}
\end{figure}\\
{\it Case $2$:} $b(\alpha_{s_0})=b(\alpha_{t_0})<e(\alpha_{s_0})< e(\alpha_{t_0})$. \\
In this case, we have $s_0\cap t_1=s_0\cap s_1\cap t_0\cap t_1=\{b(\alpha_{s_0})\}$. By a similar argument as in Case $1$, we conclude that $Z=[s_0,t_1]$ is $W$-exposed. The arguments for the remaining cases can be obtained by symmetry from the ones above and have similar arguments. The lemma then follows.
\end{proof}

\begin{defn}
Let $X$ be a CC of a non-overlapping bi-secondary structure $R=(S,T)$. We call $$C(X)=\{\ Y_{\delta}\in K_2(R)^*|\gamma(Y)\in X\}$$ the $\it{closure}$ of $X$.
\end{defn}

\begin{lemma}\label{nocopy}
Let $X$ be a CC of a non-overlapping bi-secondary structure $R=(S,T)$. Then, for all $Y_p,Y'_q\in C(X)$ we have $Y=Y'\implies p=q$. I.e. the closure of a crossing component does not contain two copies of the same $2$-simplex.
\end{lemma}
\begin{proof}
Let $ Y_p,Y_q\in C(X)$ with $Y=[x,y,z]\in K_2(R)$ and where w.l.o.g. we can assume that $[x,y]$ is the pure edge of $Y$. Since $R$ is non-overlapping we must have that $\omega(Y)=1$. This means that $[\{b(\alpha_x)\}= \Omega(Y)]\vee[\{e(\alpha_x)\}=\Omega(Y)]$. Thus $[p=b(\alpha_x)=q\in \Omega(Y)]\vee[p=e(\alpha_x)=q\in\Omega(Y)]$. In either case the lemma follows.
\end{proof}

\begin{lemma}\label{disjoint}
Let $X,X'$ be two distinct CCs of the non-overlapping bi-secondary structure $R=(S,T)$. Then $C(X)\cap C(X')=\varnothing$.
\end{lemma}
\begin{proof}
Suppose $Y_p\in C(X)\cap C(X')$. Then there exists $\gamma(Y)\in X,\gamma'(Y)\in X'$ with $p=b(\gamma(Y))$ or $p=e(\gamma(Y))$ and similarly $p=b(\gamma'(Y))$ or $p=e(\gamma'(Y))$. Combining either of the cases would imply either:\\
The two arcs $\gamma(Y)$ and $\gamma'(Y)$ share $p$ as an endpoint - a contradiction, since by hypothesis, $R$ is non-overlapping and hence has no nucleotides of degree four in its arc diagram.\\
Or: it would imply the fact that $\gamma(Y)=\gamma'(Y)$. But then $X\cap X'\ne\varnothing$. By defintion of CCs we must then conclude that $X=X'$ - another contradiction to the hypothesis. Thus it must be that $C(X)\cap C(X')=\varnothing$, and so the lemma follows.
\end{proof}

\section{Closures and Spheres}

\begin{lemma}\label{glue}
Let $R=(S,T)$ be a non-overlapping bi-secondary structure and let $C(X)$ be the closure of a CC $X$ of $R$. Then, for any $Y_p\in C(X)$ and any $1$-face $[u,v]$ of $Y$, there exists $Y'_q\in C(X)$, $Y_p\ne Y'_q$, with $Y_p\cap Y'_q=[u,v]$. I.e. any decoration (triangle) in $C(X)$ is glued along all of its $1$-faces (edges) to decorations still in $C(X)$. Furthermore, the only decorations in $C(X)$ that have as a face the edge $[u,v]$ are $Y_p$ and $Y'_q$.
\end{lemma}
\begin{proof}
Let $N(X)=\{\delta|Y_{\delta}\in C(X)\}$ be the set of nucleotides that index the decorations in the closure of the CC $X$ of $R$. We can introduce a cyclical ordering on $N(X)$ by letting $p\in N(X)$ precede $q\in N(X)$ if $q$ is the smallest nucleotide such that $p<q$. Furthermore we set $\max[N(X)]$ to precede $\min[N(X)]$. This cyclical order induces a cyclical order on $C(X)$ where $Y_p\in C(X)$ precedes $Y'_q\in C(X)$ if $p$ precedes $q$ in $N(X)$. By virtue of Lemma~\ref{nocopy} this order is well defined. We call this order the $\it{band}$ order of $C(X)$.

Now, w.l.o.g. let $Y_p=[x,y,z]_p\in C(X)$ be a decoration at $p$ with $\gamma(Y)\in X$ the pure arc of $Y_p$ and let the pure edge of $Y$ be $[x,y]$. For each edge of the decoration $Y_p$ we would like to identify another decoration $Y'_q\in C(X)$ that shares that edge with $Y_p$.
Firstly, clearly $Y_p$ shares the pure edge $[x,y]$ with the decoration $Y'_{e(\gamma(Y))}\in C(X)$. This is since $\gamma(Y)=\alpha_x$ and so $p=b(\gamma(Y))=b(\alpha_x)\implies q=e(\alpha_x)=e(\gamma(Y))$ (See Figure~\ref{pureedgegluefigure}). Since $[x,y]$ is the pure edge of $Y$, it can only appear as a $1$-simplex in $Y_p$ and in $Y'_{e(\gamma(Y))}$, also as its pure edge. Hence they are the only two decorations in $C(X)$ that contain $[x,y]$ as a face.

\begin{figure}[htbp]
    \centering
    \includegraphics[width=0.25\textwidth]{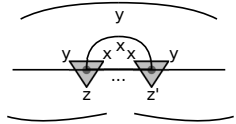}
    \caption{
      The decoration $Y_p=[x,y,z]_p$ with $p=b(\alpha_x)$, is glued along its pure edge $[x,y]$ to the decoration $Y'_q=[x,y,z']_q$ with $q=e(\alpha_x)$. Note that in this case $\gamma(Y)=\gamma(Y')=\alpha_x$.}
    \label{pureedgegluefigure}
\end{figure}

Consider now a mixed edge of $Y_p$. We claim that this edge is present in the decoration $Y'_q$ that: is the predecessor OR that precedes $Y_p$ in the the band order of $C(x)$. Suppose our chosen edge is $[x,z]\subseteq Y_p$ and let $Y'_q$ succeed $Y_p$ in the band order. Note that, by definition, we must then have that $q$ is the closest (minimal) nucleotide to $p$ (w.r.t. the cyclic ordering on $N(X)$). To show that $[x,z]\subseteq Y'_q$ it suffices to note that if $r$ would be a nucleotide at which we would have a decoration $Y''_r$, and said nucleotide would be in between $p$ and $q$ then, we must have $\forall Y''_r\in K_2(R)^*\implies Y''_r\not\in C(X)$. Otherwise $r$ would violate the minimality of $q$ (See Figure ~\ref{mixededgegluefigure}).

\begin{figure}[htbp]
    \centering
    \includegraphics[width=0.25\textwidth]{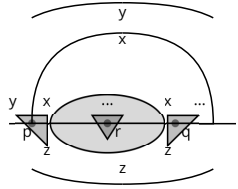}
    \caption{
      The decoration $Y_p=[x,y,z]_p$ with $p=b(\alpha_x)$, is glued along its mixed edge $[x,z]$ to the decoration $Y'_q=[x,y,z']_q$. By minimality of $Y'_q$ we must have that for any decoration $Y''_r$ with $p\le r\le q$, $Y''_r\not\in C(x)$.}
    \label{mixededgegluefigure}
\end{figure}

Hence we must have $[x,z]\subseteq Y'_q$. Now, to show that $Y'_q$ is the only other decoration in $C(X)$ that contains the face $[x,z]$ we argue as follows:

Suppose there exists another decoration $Y''_r\in C(X)$, $r\ne q$ that also contains the face $[x,z]$. Then by lemma~\ref{nocopy} we must have $Y''\ne Y'$ and so $\gamma(Y')\ne \gamma(Y'')$. Since $q$ is minimal, we must have $b(\alpha_x)\le b(\gamma(Y'))\le e(\gamma(Y'))\le b(\gamma(Y''))\le e(\gamma(Y''))\le e(\alpha_x)$ (See Figure~\ref{mixededgeuniquefigure}). 

\begin{figure}[htbp]
    \centering
    \includegraphics[width=0.25\textwidth]{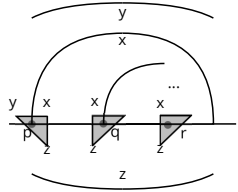}
    \caption{
      The decoration $Y_p=[x,y,z]_p$ with $p=b(\alpha_x)$, is glued along its mixed edge $[x,z]$ to the decoration $Y'_q=[x,y,z']_q$. By minimality of $Y'_q$ we must have that for any decoration $Y''_r\in C(x)$ with $[x,z]\subseteq Y''_r$, $b(\alpha_x)\le b(\gamma(Y'))\le e(\gamma(Y'))\le b(\gamma(Y''))\le e(\gamma(Y''))\le e(\alpha_x)$.}
    \label{mixededgeuniquefigure}
\end{figure}

Since $\gamma(Y')\in X$, there must exists a sequence of pairwise crossing arcs that terminates with $\alpha_x$, i.e. a path between $\gamma(Y')$ and $\alpha_x$ in the $X$-vertex induced arc line sub-graph of $R$. Note that for such arcs $w$ in this sequence we cannot have $b(w)\le b(x)\le e(w)$ otherwise the edge $[x,z]\subseteq Y_p$ would have to contain $w$ in its labeling. Hence this sequence of arcs must connect $\gamma(Y')$ to $\alpha_x$ through an arc $w'$ such that $b(w')\le e(x)\le e(w')$. However, since $e(\gamma(Y'))\le b(\gamma(Y''))\le e(\gamma(Y''))\le e(\alpha_x)$ then, either $b(w'')\le b(\gamma(Y''))\le e(\gamma(Y''))\le e(w'')$ for some $w''$ in the sequence, or at the very least $b(w'')\le b(\gamma(Y''))\le e(w'')$. In either case however, the label of the edge $[x,z]\subseteq Y''_r$ would have to contain $w''$. Since $[x,z]$ is fixed, so is its labeling, and hence a contradiction arises. This show that there does not exists another decoration $Y''_r\in C(x)$ with $[x,z]$ as a face. 

A similar argument holds for the edge $[x,z]\subseteq Y_p$, and thus the lemma follows. 
\end{proof}

\begin{lemma}\label{issphere}
Let $R=(S,T)$ be a non-overlapping bi-secondary structure and let $C(X)$ be the closure of a CC $X$ of $R$. There exists a Euclidean $3$-space embedding of $C(X)$ that is homeomorphic to a $2$-sphere.
\end{lemma}
\begin{proof}

By Lemma~\ref{glue} and Lemma~\ref{nocopy} we can conclude that there exists a Euclidean $3$-space embedding of $C(X)$ that is a closed surface. It suffices to show that this surface is a sphere. To this end we argue as follows: Let $P$ be the triangulated annular region obtained by the pairwise consecutive gluing of the decorations in $C(X)$ following the band order, only along edges that are mixed (See Figure~\ref{annulargluefigure}).

\begin{figure}[htbp]
    \centering
    \includegraphics[width=0.75\textwidth]{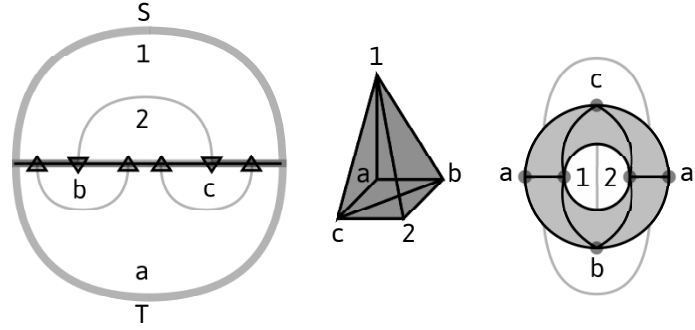}
    \caption{LHS: a bi-secondary structure with one CC, $X=\{\alpha_b,\alpha_c,\alpha_2\}$, and the CC's closure in terms of corresponding decorations. MS: The closure as a triangulation of a $2$-sphere in $K(R)$. RHS: the triangulation of the annular region $P$ with the gluing arcs corresponding to the arcs in the $CC$.}
    \label{annulargluefigure}
\end{figure}

We draw a "gluing" arc between two pure edges in $P$ if they are to be glued. It suffices to show that these arcs can be embedded in $\mathbb{R}^2\setminus P$ without crossing. The $\mathbb{R}^2\setminus P$ embedding is given by the fact that, as mentioned in the proof of Lemma~\ref{glue}, pure edges of a decoration at the endpoint of a given gluing arc will be glued to pure edges of a decoration at the other endpoint of the gluing arc. Hence, the "gluing" arcs are actually the pure arcs themselves. Furthermore the pure-arcs corresponding to the inside boundary of $P$ will be arcs from the secondary structure $S$ while those corresponding to the outside boundary of $P$ correspond to arcs in $T$. Since $R=(S,T)$ is a bi-secondary the pure arcs will thus have a planar embedding into $\mathbb{R}^2\setminus P$ by virtue of the planarity of $R=(S,T)$. Hence the lemma follows.
\end{proof}

\begin{rmk}
Lemma~\ref{issphere} and Lemma~\ref{disjoint} allow us to immediately conclude that $$|\chi(R)|\le r(H_2(R))$$ in the case where $R=(S,T)$ is a non-overlapping bi-secondary structure. This prompts the natural question as to whether or not we actually have strict equality in the above relation. As we shall see in the following, that will indeed be the case.
\end{rmk}

\section{The Tree of Irreducible Components}

\begin{defn}
We call a (potentially trivial) connected component of the line graph of $R$, an \emph{irreducible component} (IC) of $R$.   
\end{defn}

\begin{rmk}
By definition, an IC is either a non-crossing arc in $R$, or a CC in $R$. Hence, any bi-secondary structure $R$ can be uniquely decomposed into disjoint ICs.
\end{rmk}

\begin{defn}
Let $R = (S,T)$ be a non-overlapping bi-secondary structure. Let $X$ and $X'$ be two distinct ICs of $R$. Then we say $X$ is \emph{nested by} $X'$ which we denote by $X\ll X'$, if and only if there exists an arc $\epsilon'\in X'$, such that for all $\epsilon\in X$, we have $\epsilon\prec_S \epsilon'$ or $\epsilon\prec_T \epsilon'$.
\end{defn}

\begin{rmk}\label{ICpartition}
Clearly, the $\ll$ relation defines a poset structure on the the set of ICs of $R$. As a result, a bi-secondary structure can be constructed from ICs via nesting and concatenation. Hence, each IC has a unique cover (parent) w.r.t the $\ll$ poset order.
\end{rmk}

Below, we extend in a natural fashion, the definition of the closure of a CC to that of the closure of an IC. 

\begin{defn}
Let $X$ be an IC of a non-overlapping bi-secondary structure $R=(S,T)$. We call
$$
C(X)=\{Y_{\delta}\in K_2(R)^*|\gamma(Y)\in X\}
$$
the closure of $X$.
\end{defn}

The $\ll$ poset order induces a tree-like structure over all the sub-simplicial complexes generated by the closures of the ICs. Let $\langle C(X)\rangle$ denote the sub-simplicial complex of $K(R)$ generated by $\{Y|\gamma(Y)\in X\}$. Lemma~\ref{issphere} shows that when $X$ is a CC, $\langle C(X)\rangle$ is homeomorphic to a $2$-sphere. We first show that when $X$ is a trivial IC, i.e., $X$ contains only $1$ arc, $\langle C(X)\rangle$ is a single $2$-simplex (triangle).

\begin{lemma}\label{singletriangle}
Let $X$ be a trivial IC of a non-overlapping bi-secondary structure $R=(S,T)$. Then $\langle C(X)\rangle$ is a $2$-simplex.
\end{lemma}
\begin{proof}
W.l.o.g, we can assume $X=\{\mu\}$, where $\mu\in S$. Let $\epsilon$ be the cover of $\mu$ w.r.t. $\prec_S$. Let $\beta$ be the cover of $\mu$ w.r.t. $\prec_T$ (when $\mu$ is flipped to the $T$ side of the diagram). Since $\mu\in S$ does not cross an arc in $T$ we must w.l.o.g. have $b(\beta)\le b(\mu)\le e(\mu)\le e(\beta)$. Let $Y=[s_0,s_1,t]$ with $\alpha_{s_0}=\mu,\alpha_{s_1}=\epsilon,\alpha_t=\beta$. Then $Y_{b(\mu)}=Y_{e(\mu)}$ and so we must have $C(X)=\{Y_{b(\mu)},Y_{e(\mu)}\}=\{Y\}$ (See Figure~\ref{singletrianglefigure}).

\begin{figure}[htbp]
    \centering
    \includegraphics[width=0.25\textwidth]{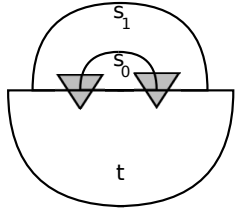}
    \caption{Here $\mu=\alpha_{s_0}$. The decorations at $b(\alpha_{s_0})$ and $e(\alpha_{s_0})$ come from the same $2$-simplex $Y=[s_0,s_1,t]$.}
    \label{singletrianglefigure}
\end{figure}

Hence, the lemma follows. 
\end{proof}

\begin{defn}
Let $X$ be an IC of a non-overlapping bi-secondary structure $R=(S,T)$. We say $\epsilon$ is the minimal $S$-arc that nests $X$ and $\beta$ is the minimal $T$-arc that nests $X$ if and only if $\forall \mu\in X, \epsilon\prec_S \mu \prec_T\beta$ (when $\mu$ is flipped to the $S$ and $T$ sides of the diagram respectively). Two such arcs always exist since $R$ is a bi-secodary structure. The $1$-simplex $\{s,t\}\in K_1(R)$ with $\alpha_s=\epsilon,\alpha_t=\beta$ is called the \emph{up (mixed) edge} of $\langle C(X)\rangle$. All other mixed $1$-simplices of $\langle C(X)\rangle$ are called \emph{down (mixed) edges} of $\langle C(X)\rangle$.
\end{defn}

\begin{rmk}
The up edge of $C(X)$ is in fact present as a $1$-simplex of the complex $\langle C(X)\rangle$. To see this, it suffices to show that the up edge is a $1$-face of at least one decoration in $C(X)$.
\end{rmk}
\begin{proof}
Let $p=\min N(X)$. W.l.o.g., we can assume $p$ is an end point of an arc $\epsilon'$ in $S$. Note that $b(\epsilon)\le p\le e(\epsilon)$ and Similarly, $b(\beta)\le p\le e(\beta)$. Let $Y=\{s_0,s_1,t\}$ with $\alpha_{s_0}=\epsilon',\alpha_{s_1}=\epsilon,\alpha_t=\beta$. Then we must have $Y_p\in C(X)$ and the remark follows (See Figure~\ref{upedgeinclosurefigure}).

\begin{figure}[htbp]
    \centering
    \includegraphics[width=0.5\textwidth]{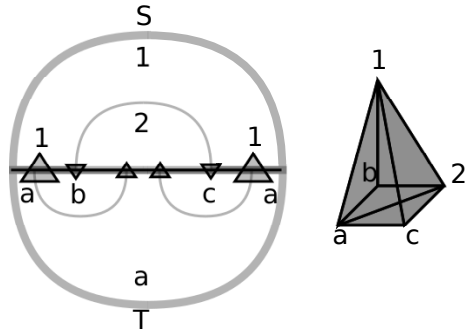}
    \caption{LHS: a bi-secondary structure with a single CC $X=\{\alpha_b,\alpha_c,\alpha_2\}$. The $1$-simplex $[1,a]$ is the up-edge of the CC. RHS: The closure of the LHS's CC.}
    \label{upedgeinclosurefigure}
\end{figure}

\end{proof}

\begin{rmk}\label{down}
Let $X$ be an IC of a non-overlapping bi-secondary structure $R=(S,T)$. Let $\epsilon$ be the minimal $S$-arc that nests $X$ and let $\beta$ be the minimal $T$-arc that nests $X$.
Let $\{s',t'\}$ be a down edge of $\langle C(X)\rangle$, and let $\alpha_{s'}$ and $\alpha_{t'}$ be the corresponding maximal arcs of the $s'$ and $t'$ loops respectively. Based on the annular construction in the proof of Lemma~\ref{issphere}, we know $\{\alpha_{s'},\alpha_{t'}\}\subset X\bigcup \{\epsilon,\beta\}$. Furthermore, the set $\{\alpha_{s'},\alpha_{t'}\}$ contains at most one arc from the set $\{\epsilon,\beta\}$. Hence $\{\alpha_{s'},\alpha_{t'}\}$ contains at least one arc from $X$.
\end{rmk}

We next reveal the tree-like structure of $\{\langle C(X)\rangle\}_{X=IC}$ mentioned above. We do this by investigating the poset order $\ll$ over the set of all ICs. 

\begin{lemma}\label{updown}
Let $X$ be an IC of a non-overlapping bi-secondary structure $R=(S,T)$ and let $\{s,t\}$ be the up edge of $\langle C(X)\rangle$, where $s\in S$, $t\in T$ and $\alpha_s=\epsilon$, $\alpha_t=\beta$. Then $X'$, the cover of $X$ under the $\ll$ poset order, is the unique IC such that $\langle C(X')\rangle$ contains $\{s,t\}$ as a down edge.
\end{lemma}
\begin{proof}

We distinguish the following two cases:\\

Case 1: $\epsilon$ and $\beta$ are contained in the same IC.
In this case, both $\epsilon$ and $\beta$ are contained in $X'$, since $\epsilon$ and $\beta$ are the minimal $S$-arc and $T$-arc respectively that both nest $X$. Let $p'$ be the largest nucleotide in $N(X')$ that is smaller than the smallest nucleotide $p\in N(X)$. The decoration $Y_{p'}\in C(X')$ thus contains $[s,t]$ as a mixed edge for $\alpha_s=\epsilon$ and $\alpha_t=\beta$. Since both $\epsilon$ and $\beta$ are contained in $X'$, by Remark~\ref{down}, $[s,t]$ is a down mixed edge of $\langle C(X')\rangle$. Furthermore, since $X'$ is the unique IC that contains $\epsilon$ and $\beta$, $X'$ is the unique IC such that $\langle C(X')\rangle$ contains $[s,t]$ as a down mixed edge. \\

Case 2: $\epsilon$ and $\beta$ are contained in different ICs. In this case, $\epsilon$ and $\beta$ must be nested within one another. W.l.o.g. we can assume $b(\beta)\le b(\epsilon)\le e(\epsilon)\le e(\beta)$, i.e. $\epsilon$ is nested by $\beta$. Then $\epsilon$ is contained in $X'$. Since $\epsilon$ and $\beta$ are contained in different ICs, $\beta$ must be the minimal $T$-arc that also nests $X'$.

Let $p'$ be the largest nucleotide in $N(X')$ that is smaller than the smallest nucleotide $p\in N(X)$.The decoration $Y_{p'}\in C(X')$ thus contains $[s,t]$ as a mixed edge for $\alpha_s=\epsilon$ and $\alpha_t=\beta$. Since $\epsilon$ is contained in $X'$, by Remark~\ref{down}, $[s,t]$ is a down mixed edge of $\langle C(X')\rangle$. On the other hand, let $X''$ be the IC that contains $\beta$, since $\epsilon$ is nested by $\beta$, $[s,t]$ can not be a $1$-face of $\langle C(X'')\rangle$ (See Figure~\ref{updowncase2figure}).

\begin{figure}[htbp]
    \centering
    \includegraphics[width=0.25\textwidth]{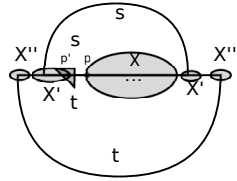}
    \caption{Here $\epsilon=\alpha_s$, $\beta=\alpha_t$ and $X''\ll X'\ll X$.}
    \label{updowncase2figure}
\end{figure}

Since $X'$ is the unique IC that contains $\epsilon$, $X'$ is the unique IC such that $\langle C(X')\rangle$ contains $[s,t]$ as a down mixed edge. Hence the lemma follows.
\end{proof}

\begin{rmk}
Lemma~\ref{updown} shows us that $\{\langle C(X)\rangle\}_{X=IC}$, and hence $K(R)$, has a tree-like structure (See Figure~\ref{treelikestructurefigure}).

\begin{figure}[htbp]
    \centering
    \includegraphics[width=1.0\textwidth]{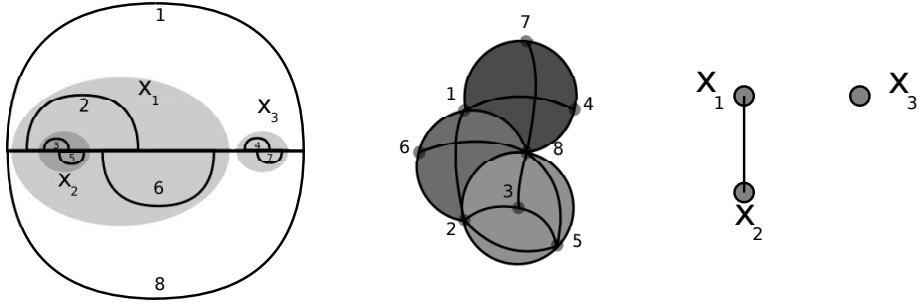}
    \caption{LHS: a bi-secondary $R$ structure with tree CCs, $X_1=\{\alpha_2,\alpha_6\},X_2\{\alpha_3,\alpha_5\}$ and $X_3=\{\alpha_4,\alpha7\}$. MS: the CC spheres in $K(R)$. RHS: the tree-like structure of $\{\langle C(X)\rangle\}_{X=IC}$. RHS: Note that only $X_2\ll X_1$, while $\langle C(X_1)\rangle$ and $\langle C(X_3)\rangle$ share the mixed up edge $[1,8]$.}
    \label{treelikestructurefigure}
\end{figure}

\end{rmk}

\section{Crossing Components and Homology Ranks for Non-overlapping Bi-structures}

\begin{thm}\label{thm-non-overlap}
Let $R=(S,T)$ be a non-overlapping bi-secondary structure. Let $r(H_2(R))$ denote the rank of the second homology group of $K(R)$. Then $$r(H_2(R))=|\chi(R)|.$$
\end{thm}
\begin{proof}
The basic idea behind this proof is to recursively decompose $Ker(\partial_2)$, following the tree-like structure of $K(R)$ such that each CC will contribute exactly one basis vector to $Ker(\partial_2)$.

Since $R$ is a non-overlapping bi-secondary structure, $K_3(R)=\varnothing$. Therefore $Im(\partial_3)=0$ and thus $H_2(R)\cong Ker(\partial_2)$. Let us consider $\tau\in Ker(\partial_2)$ where
$$
\tau=\sum_{Y\in K_2(R)}n_Y Y.
$$
Note that for each $Y$, its corresponding pure arc $\gamma(Y)$ is either crossing or non-crossing. Furthermore, if $\gamma(Y)$ is crossing, then it must be contained in exactly one of the CCs by definition. Assume $|\chi(R)|=k$ and let $X_1,X_2,\ldots,X_k$ be the CCs of $R$. We can further decompose $\tau$ into the following sum
$$
\tau=\sum_{\gamma(Y)\; non-crossing}n_Y Y+\sum_{j=1}^k\sum_{\gamma(Y^j)\in X_j}n_{Y^j}Y^j.
$$
Since $\tau\in Ker(\partial_2)$, we have

$$
\partial_2(\tau)=\sum_{\gamma(Y)\; non-crossing}n_Y\partial_2(Y)+\sum_{j=1}^k\partial_2(\sum_{\gamma(Y^j)\in X_j}n_{Y^j} Y^j)$$
$$
=\sum_{\gamma(Y)\; non-crossing}n_Y\overline{Z^P}+\sum_{\gamma(Y)\; non-crossing}n_Y(\overline{Z_1^M}+\overline{Z_2^M})+$$
$$+\sum_{j=1}^k\partial_2(\sum_{\gamma(Y^j)\in X_j}n_{Y^j}Y^j)=0,
$$
where $\overline{Z^P}$ and $\overline{Z_{1,2}^M}$ are the signed pure $1$-faces and the mixed $1$-faces of $Y$ respectively, such that $\gamma(Y)$ is non-crossing. By Lemma~\ref{exposededge}, we know that for all non-crossing arcs $\gamma(Y)$, $\overline{Z^P}$ is exposed. Thus the coefficient of $\overline{Z^P}$ in $\partial_2(\tau)$ is $n_Y$. Since $\partial_2(\tau)=0$, we must have $n_Y=0$. Thus, in the expression of $\partial_2(\tau)$, the sum over non-crossing arcs disappears.

Next, we will focus on the term $j\partial_2(\sum_{\gamma(Y^1)\in X_1}n_{Y^1}Y^1)$ in the expression of $\partial_2(\tau)$, where $X_1$ is a CC that is minimal w.r.t. $\ll$ among all other CCs of $R$ (i.e. $X_1$ does not nest any other CC of $R$). We will rewrite this term as a linear combination of $1$-faces of $\langle C(X_1) \rangle$ while further partitioning said linear combination based on the types of $1$-faces in $\langle C(X_1) \rangle$, namely, pure, down mixed and up mixed
$$
\partial_2(\sum_{\gamma(Y^1)\in X_1}n_{Y^1}Y^1)=\sum_{Z^P\in \langle C(X_1) \rangle\; pure}m_{Z^P}Z^P+$$
$$+\sum_{Z^D\in \langle C(X_1) \rangle \;down \;mixed}m_{Z^D}Z^D+m_{Z^U} Z^U.
$$
The first sum is taken over all pure $1$-faces $Z_p$ of $\langle C(X_1) \rangle$.  The second sum is taken over all down mixed $1$-faces of $\langle C(X_1) \rangle$. The last term corresponds to the unique up mixed edge of $\langle C(X_1)\rangle$.

Let us examine the first sum. Note that each pure edge of $K_1(R)$ corresponds to a unique arc in $R$, namely, the pure arc of any decoration that contains said pure edge (see Remark~\ref{purearc}). By Remark~\ref{ICpartition} we can conclude that for any $Z^P$, $X_1$ is the unique IC such that $\langle C(X_1)\rangle$ contains $Z^P$ as a pure edge. Therefore, the coefficient of $Z^P$ in $\partial_2(\tau)$ is $m_{Z^P}$. Since $\partial_2(\tau)=0$, we must have $m_{Z^P}=0$. Hence, the first sum in the decomposition of $\partial_2(\sum_{\gamma(Y^1)\in X_1}n_{Y^1}Y^1)$ disappears.

Now, for the second sum, since $X_1$ is a CC that does not nest any other CC in $R$, by Lemma~\ref{updown}, each $Z^D$ is either: the up edge of some $Y^{Z^D}$ where $\gamma(Y^{Z^D})$ is non-crossing, OR it is not contained in any other $\langle C(X') \rangle$ for $X'$ another CC of $R$. We can then conclude that the coefficient of $Y^{Z^D}$ in $\tau$ must be zero, since if $\gamma(Y^{Z^D})$ is non-crossing then its coefficient in $\tau$ must be $0$ by the argument above regarding the first sum. Therefore, regardless, the coefficient of $Z^D$ in $\partial_2(\tau)$ is $m_{Z^D}$. Since $\partial_2(\tau)=0$, we must have $m_{Z^D}=0$. Hence, the second sum in the decomposition of $\partial_2(\sum_{\gamma(Y^1)\in X_1}n_{Y^1}Y^1)$ disappears.

We can thus conclude that

$$\partial_2(\sum_{\gamma(Y^1)\in X_1}n_{Y^1}Y^1)=m_{Z^U} Z^U.$$

Note however that
$$
0=\partial_1(\partial_2(\sum_{\gamma(Y^1)\in X_1}n_{Y^1}Y^1))=m_{Z^U}\partial_1(Z^U).
$$
Since $K(R)$ is a simplicial complex, each of its $1$-faces contains two distinct $0$-faces. Therefore, $\partial_1(Z^U)\neq 0$. As a result, we must have $m_{Z^U}=0$. Hence we can conclude that if $\tau\in Ker(\partial_2)$ then $$\partial_2(\sum_{\gamma(Y^1)\in X_1}n_{Y^1}Y^1)=0.$$

We now apply the above argument recursively, from bottom to top, following the $\ll$ poset order on the CCs of $R$. Thus, for each CC $X_j\in R$, we will eventually have  
$$\partial_2(\sum_{\gamma(Y^j)\in X_j}n_{Y^j} Y^j)=0.$$  Since for each CC $X_j\in R$, $\langle C(X_j) \rangle$ is a triangulation of a $2$-sphere, by \citep{hatcher2005algebraic}, $$H_2(\langle C(X_j) \rangle)\cong \mathbb{Z}.$$
Thus, there exists $V_j=\sum_{\gamma(Y^j)\in X_j}v_{Y^j}Y^j$, such that $\sum_{\gamma(Y^j)\in X_j}n_{Y^j}Y^j$ can be uniquely represented as $l_j V_j$, for some $l_j\in \mathbb{Z}$. Furthermore, By Lemma~\ref{disjoint}, all closures $C(X_j)$ for $X_j$ a CC of $R$ are disjoint. Thus $\{V_j\}_{1\leq j\leq k}$ are linearly independent. Therefore, any $\tau\in Ker(\partial_2)$ can be uniquely represented as 
$\tau=\sum_{j=1}^k l_j V_j$. As a result, we have $$H_2(R)\cong Ker(\partial_2)\cong \mathbb{Z}^k=\mathbb{Z}^{|\chi(R)|}\implies r(H_2(R))=|\chi(R)|,$$ and the theorem follows.
\end{proof}

\section{Scoops, Splits and Homology Ranks for arbitrary Bi-structures}

Let $R(S,T)$ be a bi-secondary structure over $[n]$ and let $$P=\{p\in \{1,\ldots,n\}| deg(p)=4 \text{ in the arc diagram of } R\}.$$  The two arcs that meet at $p$ determine four mutually intersecting loops $s_0,s_1,t_0,t_1$ which contribute a unique $3$-simplex $W\in K_3(R)$ to the simplicial complex $K(R)$ (See Figure~\ref{filledtetrahedrafigure}).

\begin{figure}[htbp]
    \centering
    \includegraphics[width=0.5\textwidth]{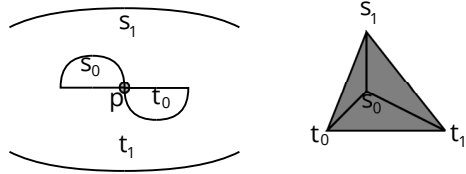}
    \caption{LHS: $s_0\cap s_1\cap t_0\cap t_1=\{p\}$. RHS: the $3$-simplex $W=[s_0,s_1,t_0,t_1]$.}
    \label{filledtetrahedrafigure}
\end{figure}

Lemma~\ref{existenceofmixed} guarantees that, among the $1$-faces of the simplex $W$, at least one of them, call it $Z\in K_1(R)$, is $W$-exposed. W.l.o.g. we can assume that $Z=[s_1,t_0]$.

\begin{defn}
Let $\mathcal{R}_p$ be a topological retraction $$\mathcal{R}_p:K(R)\longrightarrow \overline{K(R)}$$
where $\overline{K(R)}=\dot\bigcup_{d=0}^{\infty} \overline{K_d(R)}$ is the induced topological space of the simplicial complex obtained by removing the $1$-simplex $Z$ and all subsequent higher dimensional simplices of $K(R)$ that have $Z$ as a face. Namely, 
$$\overline{K_0(R)}=K_0(R),\overline{K_1(R)}=K_1(R)\setminus\{Z\},$$$$\overline{K_2(R)}=K_2(R)\setminus\{[s_0,s_1,t_0],[s_1,t_0,t_1]\},$$
$$\overline{K_3(R)}=K_3(R)\setminus\{W\},\overline{K_d(R)}=K_d(R)\text{ for all }d\ge 4.$$
We call $\mathcal{R}_p$ the scoop of $R$ at $p$ (See Figure~\ref{scoopfigure}).

\begin{figure}[htbp]
    \centering
    \includegraphics[width=0.5\textwidth]{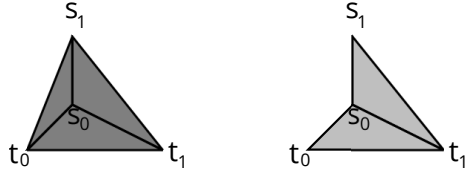}
    \caption{LHS: (before the scoop) the $3$-simplex $W=[s_0,s_1,t_0,t_1]$. RHS: (after the scoop) removing the $1$-simplex $Z=[s_1,t_0]$ and all higher dimensional simplices that contain it as a face, we are left with the two $2$-simplices $[s_0,s_1,t_1]$ and $[s_0,t_0,t_1]$.}
    \label{scoopfigure}
\end{figure}

\end{defn}

\begin{rmk}\label{scoop-iso}
 Since for each $p\in P$, $\mathcal{R}_p$ is a retraction, we can immediately conclude that
$$H_2(\circ_{p\in P}\mathcal{R}_p(K(R)))\cong H_2(R).$$    
\end{rmk}

\begin{defn}
Let $\mathcal{S}_p$ be a mapping that takes the bi-secondary structure $R$ over $[n]$ to the bi-secondary structure $R'$ over $[n+1]$ by splitting the nucleotide $p$ into two adjacent nucleotides $q_1,q_2$ such that the arcs in $R$ that have one endpoint at $p$ now have endpoints at $q_1$ and $q_2$ respectively and do not cross. We call $\mathcal{S}_p$ a split of $R$ at $p$. 
\end{defn}

\begin{rmk}
For each $p\in P$, it is immediately clear that such a mapping $\mathcal{S}_p$ always exists.
\end{rmk}

\begin{lemma}\label{scoopsplit}
Let $R(S,T)$ be a bi-secondary structure over $[n]$ and let $P$ be defined as above. Furthermore let $p\in P$ be fixed. Then,
$$K(S_p(R))\cong R_p(K(R)).$$
I.e. the simplicial complex of $R$ split at $p$,  is homeomorphic as a topological space to the scoop of $R$ at $p$.
\end{lemma}
\begin{proof}
Let $W=[s_0,s_1,t_0,t_1]\in K_3(R)$, with $s_0\le  s_1\le t_0\le t_1$ (in terms of the simplicial ordering on $K(R)$) be the $3$-simplex determined by the two arcs that meet at $p$. Since $\{p\}\subseteq s_0\cap s_1\cap t_0\cap t_1$, $\alpha_{s_0}$ and $\alpha_{t_0}$ must share at least one endpoint. W.l.o.g., we distinguish the following two cases (See Figure~\ref{existenceofmixedfigure}):\\

{\it Case $1$:} $b(\alpha_{s_0})<e(\alpha_{s_0})=b(\alpha_{t_0})<e(\alpha_{t_0})$.\\
In this case, after splitting $R$ at $p$, we obtain $b(\alpha_{\overline{s_0}})<e(\alpha_{\overline{s_0}})<b(\alpha_{\overline{t_0}})<e(\alpha_{\overline{t_0}})$ with the new loops $\overline{s_0}=(s_0\setminus\{p\})\cup\{q_1\}$, $\overline{t_0}=(t_0\setminus\{p\})\cup\{q_2\}$, $\overline{s_1}=(s_1\setminus\{p\})\cup\{q_1,q_2\}$ and finally $\overline{t_1}=(t_1\setminus\{p\})\cup\{q_1,q_2\}$.\\
Note that, $\overline{s_1}\cap x\ne\varnothing\Leftrightarrow s_1\cap x\ne\varnothing,\forall x\in R$ and $\overline{t_1}\cap x\ne\varnothing\Leftrightarrow t_1\cap x\ne\varnothing,\forall x\in R$. Also, $\overline{s_0}\cap x\ne\varnothing\Leftrightarrow s_0\cap x\ne\varnothing,\forall x\in R\setminus\{t_0\}$ and $\overline{t_0}\cap x\ne\varnothing\Leftrightarrow t_0\cap x\ne\varnothing,\forall x\in R\setminus\{s_0\}$.Finally, $\overline{s_0}\cap\overline{t_0}=\varnothing$. Thus, in this case we must have $K(S_p(R))\cong R_p(K(R))$.(See Figure~\ref{splitfigure}).

\begin{figure}[htbp]
    \centering
    \includegraphics[width=0.5\textwidth]{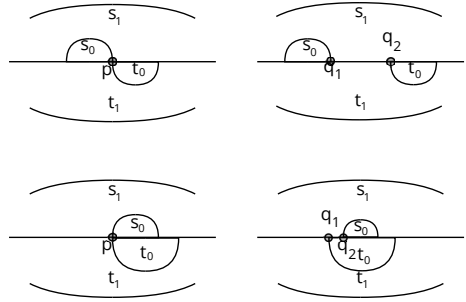}
    \caption{LHS: before the split. RHS: after the split. TOP: Case 1 split. BOTTOM: Case 2 split.}
    \label{splitfigure}
\end{figure}

{\it Case $2$:} $b(\alpha_{s_0})=b(\alpha_{t_0})<e(\alpha_{s_0})<e(\alpha_{t_0})$. \\
In this case, after splitting $R$ at $p$, we obtain $b(\alpha_{\overline{t_0}})<b(\alpha_{\overline{s_0}})<e(\alpha_{\overline{s_0}})<e(\alpha_{\overline{t_0}})$ with the new loops $\overline{s_0}=(s_0\setminus\{p\})\cup\{q_2\}$, $\overline{t_0}=(t_0\setminus\{p\})\cup\{q_1\}$, $\overline{s_1}=(s_1\setminus\{p\})\cup\{q_1,q_2\}$ and finally $\overline{t_1}=(t_1\setminus\{p\})\cup\{q_1\}$.\\
Note that, $\overline{s_1}\cap x\ne\varnothing\Leftrightarrow s_1\cap x\ne\varnothing,\forall x\in R$ and $\overline{t_1}\cap x\ne\varnothing\Leftrightarrow t_1\cap x\ne\varnothing,\forall x\in R\setminus\{s_0\}$. Also, $\overline{s_0}\cap x\ne\varnothing\Leftrightarrow s_0\cap x\ne\varnothing,\forall x\in R\setminus\{t_1\}$ and $\overline{t_0}\cap x\ne\varnothing\Leftrightarrow t_0\cap x\ne\varnothing,\forall x\in R$. Finally, $\overline{s_0}\cap\overline{t_1}=\varnothing$. Hence in this case as well, we must have $K(S_p(R))\cong R_p(K(R))$.\\

The arguments for the remaining cases can be obtained by symmetry from the ones above and have similar arguments. The lemma then follows.
\end{proof}

Finally, we are in the position to prove the main result of this paper.

\begin{thm}
Let $R=(S,T)$ be an arbitrary bi-secondary structure. Then $$r(H_2(R))=|\chi(R)|.$$
\end{thm}
\begin{proof}
Denote by $R'=\circ_{p\in P}S_P(R)$ the bi-secondary structure obtained by sequential splits of $R$ at all nucleotides $p\in P$ where $P$ is defined as above. By Lemma~\ref{scoopsplit} we must have that $$K(R')\cong\circ_{p\in P}R_p(K(R)).$$ From this homeomorphism we obtain $$H_2(K(R'))\cong H_2(\circ_{p\in P}R_p(K(R))).$$ By Remark~\ref{scoop-iso} $$H_2(\circ_{p\in P}\mathcal{R}_p(K(R)))\cong H_2(R).$$ Hence $H_2(R)\cong H_2(R')$. Now $R'$ is non-overlapping since each nucleotide of degree four in the arc diagram of $R$ has been split into two nucleotides each of degree three in the arc diagram of $R'$. Thus, by Theorem~\ref{thm-non-overlap}, we have that $r(H_2(R'))=|\chi(R')|$. Finally, since each split introduces no new crossing arcs in $R'$, the number of crossing components is conserved under splitting. Hence, we must have that $|\chi(R')|=|\chi(R)|$. Thus $$r(H_2(R))=r(H_2(R'))=|\chi(R')|=|\chi(R)|$$ and the theorem follows.
\end{proof}


\section{Declarations of interest}
None.

\section{Acknowledgments}
We gratefully acknowledge the comments from Fenix Huang. Many thanks to Thomas Li, Ricky Chen and Reza Rezazadegan for discussions.

\section*{References}


\begin{thebibliography}{1}
\expandafter\ifx\csname url\endcsname\relax
  \def\url#1{\texttt{#1}}\fi
\expandafter\ifx\csname urlprefix\endcsname\relax\def\urlprefix{URL }\fi
\expandafter\ifx\csname href\endcsname\relax
  \def\href#1#2{#2} \def\path#1{#1}\fi

\bibitem{bura2019loop}
A.~C. Bura, Q.~He, C.~M. Reidys, Loop homology of bi-secondary structures,
  arXiv preprint arXiv:1904.02041.

\bibitem{hatcher2005algebraic}
A.~Hatcher, Algebraic topology, Tsinghua University Press, 2005.

\end{thebibliography}

\end{document}